\documentclass[10pt]{article}
\usepackage{amsmath, amsthm, amssymb, amsfonts, mathrsfs, mathtools}
\usepackage{graphicx}
\usepackage{makeidx}
\usepackage[left=2cm,right=2cm,top=2cm,bottom=2cm]{geometry}
\usepackage{caption}
\usepackage{subcaption}
\usepackage[section]{placeins}
\usepackage{enumitem}
\usepackage{tikz}
\usepackage{stmaryrd}
\usepackage{latexsym}
\usepackage{arydshln}
\usepackage{color}
\usepackage{lineno}
\usepackage{ae}
\usepackage{url}
\usepackage{epsfig}

\usetikzlibrary{decorations.pathreplacing}
\tikzstyle{vertex}=[auto=left,circle,draw=black,fill=white, inner sep=1.5]

\usepackage{hyperref}
\hypersetup{colorlinks=true, citecolor={blue}, linkcolor=blue, filecolor=magenta, urlcolor=cyan}

\newtheorem{theorem}{Theorem}[section]

\newtheorem{lema}[theorem]{Lemma}
\newtheorem{corollary}{Corollary}[theorem]

\title{State Transfer on Unitary Cayley Graphs and Quadratic Unitary Cayley Graphs}

\author{ Akash Kalita and Bikash Bhattacharjya\\
Department of Mathematics\\
Indian Institute of Technology Guwahati, India\\
akash.kalita@iitg.ac.in, b.bikash@iitg.ac.in }

\date{}
\begin{document}
\maketitle

\vspace{-0.3in}

\begin{center}{\textbf{Abstract}}\end{center}
The unitary Cayley graph, denoted $X_n$, is the graph with vertex set ${\mathbb{Z}}_n$ such that two distinct vertices $a$ and $b$ are adjacent if $a-b=u$ for some $u$ with $1 \leq u \leq n-1$ and $\gcd(u,n) = 1$. The quadratic unitary Cayley graph, denoted $G_n$, is the graph with vertex set ${\mathbb{Z}}_n$ such that two distinct vertices $a$ and $b$ are adjacent if $a-b=u^2$ or $a-b=-u^2$ for some $u$ with $1 \leq u \leq n-1$ and $\gcd(u,n) = 1$. In this paper, we classify all $X_n$ admitting pretty good fractional. We also classify all $X_n$ that admit fractional revival. It turns out that $X_n$ admits fractional revival if and only if it admits pretty good fractional revival. Further, we classify all $G_n$ admitting periodicity. As a consequence, we obtain all $G_n$ admitting perfect state transfer. We also classify $G_n$ admitting pretty good state transfer, pretty good fractional revival and fractional revival. 
%


\section{Introduction}\label{sec 1}
A quantum spin network is represented by a simple undirected graph, where the vertices correspond to qubits and edges indicate their interactions. Farhi and Gutmann~\cite{Farhi} demonstrated that the transfer of information in quantum spin networks can be represented using continuous-time quantum walks on graphs. Bose~\cite{Bose} developed the use of continuous-time quantum walks to transmit quantum states between two locations. A graph admitting perfect state transfer from a vertex $a$ to another vertex $b$ represents a quantum spin network such that the state of the qubit at $a$ can be transferred to the state of the qubit at $b$ without any loss of information. The study of perfect state transfer was initiated by Christandl et al.~\cite{Christandl 1, Christandl 2}, and it has since drawn significant interest from researchers in algebraic graph theory. However, in a simple undirected graph, perfect state transfer occurs rarely, as indicated by Godsil~\cite{When can perfect state transfer occur}. Therefore, various generalizations of perfect state transfer have been studied in the literature. One such generalization is pretty good state transfer, which was independently introduced by Godsil~\cite{State transfer on graphs} and by Vinet and Zhedanov~\cite{Vinet} in 2012. Fractional revival~\cite{Chan FR} is another generalization of perfect state transfer that has been studied recently. Fractional revival can be used for entanglement generation in quantum spin networks. In 2021, Chan et al.~\cite{PGFR Adjacency} generalized the notions of pretty good state transfer and fractional revival, referred to as pretty good fractional revival. 

Throughout the paper, the word graph represents a finite simple undirected graph with the adjacency matrix $A$ and $\mathbf{i}$ denotes the complex number $\sqrt{-1}$. Also, ${\mathbb{R}}^+$ denotes the set of all positive real numbers. For a vertex $a$ of a graph $\Gamma$, let $\textbf{e}_{a}$ denote the column vector indexed by the vertices of $\Gamma$ such that the $a$-th entry is one and zero elsewhere. The \textit{transition matrix} of $\Gamma$, denoted $H(t)$, is defined as the following matrix-valued function
\begin{align*}
H(t):=\exp(-\mathbf{i}t A)=\sum_{j=0}^{\infty} \frac{(-\mathbf{i}t A)^j}{j!},~~~\mathrm{where}~t~\mathrm{is~a~real~number}. 
\end{align*}
The transition matrix $H(t)$ determines a \textit{continuous-time quantum walk} on $\Gamma$. The graph $\Gamma$ is said to admit \textit{fractional revival} (FR) from a vertex $a$ to another vertex $b$ if there exists $t \in {\mathbb{R}}^+$ such that   
$$H(t){\textbf{e}_{a}}=\alpha{\textbf{e}_{a}}+\beta{\textbf{e}_{b}},$$ 
where $\alpha$ and $\beta$ are complex numbers with $\beta \neq 0$, and $|\alpha|^2+|\beta|^2=1$. As a specific case, if $\alpha=0$ then $\Gamma$ is said to admit \textit{perfect state transfer} (PST) from $a$ to $b$. The graph $\Gamma$ is said to admit \textit{periodicity} if there exists $t \in {\mathbb{R}}^+$ and a complex number $\gamma$ of unit modulus such that 
$$H(t)=\gamma I,~~~\mathrm{where}~I~\mathrm{is~the~identity~matrix}.$$
The graph $\Gamma$ is said to admit \textit{pretty good fractional revival} (PGFR) from a vertex $a$ to another vertex $b$ if there exists a sequence $\{t_k\}$ in $\mathbb{R}$ such that   
$$\lim_{k\to\infty} H(t_k){\textbf{e}_{a}}=\alpha{\textbf{e}_{a}}+\beta{\textbf{e}_{b}},$$ 
where $\alpha$ and $\beta$ are complex numbers with $\beta \neq 0$, and $|\alpha|^2+|\beta|^2=1$. As a specific case, if $\alpha=0$ then $\Gamma$ is said to admit \textit{pretty good state transfer} (PGST) from $a$ to $b$.  
The five phenomena periodicity, PST, PGST, FR and PGFR fall under the general notion of what is called \textit{state transfer} on graphs.

Let $S$ be a subset of a finite group $G$ such that it does not contain the identity element of $G$ and $S=\{y^{-1}\colon y \in S\}$. Such a subset is called a \textit{connection set} in $G$. The \textit{Cayley graph} over $G$ with the connection set $S$, denoted $\mathrm{Cay}(G, S)$, is the graph whose vertex set is $G$ such that two distinct vertices $a$ and $b$ are adjacent if $a^{-1}b \in S$. As a specific case, if $G$ is the group ${\mathbb{Z}}_n$, then the Cayley graph $\mathrm{Cay}(G, S)$ is called a \textit{circulant graph}. A \textit{cycle} on $n$ vertices, denoted $C_n$, is the circulant graph $\mathrm{Cay}({\mathbb{Z}}_n, \{1, n-1\})$. Let $U(n)$ denote the set $\{u \colon 1 \leq u \leq n-1~\mathrm{and}~\gcd(u,n)=1\}$. Also, let $Q_n=\{u^2 \colon u \in U(n)\}$ and $-Q_n=\{-u^2 \colon u \in U(n)\}$. Then the \textit{unitary Cayley graph} $X_n$ is defined as the circulant graph $\mathrm{Cay}({\mathbb{Z}}_n, U(n))$. The \textit{quadratic unitary Cayley graph} $G_n$ is defined as the circulant graph $\mathrm{Cay}({\mathbb{Z}}_n, T_n)$, where $T_n = Q_n \cup (-Q_n)$. 

In the last two decades, the exploration of state transfer on Cayley graphs has received considerable attention due to the concise expressions for their eigenvalues and eigenvectors. Another advantage of this class is the ability to identify which pairs of vertices might be involved in state transfer. Basic et al.~\cite{Basic} explored PST on integral circulant graphs. Tan et al.~\cite{Tan} studied periodicity and PST on Cayley graphs over abelian groups. Pal and Bhattacharjya~\cite{Pal circulant}, and Pal~\cite{More circulant graphs, State transfer on circulant Pal} explored PGST on circulant graphs. A characterization for the existence of FR on Cayley graphs over abelian groups was provided independently by Wang et al.~\cite{J. Wang} and by Cao and Luo~\cite{Cao}. The authors in~\cite{Kalita} obtained a necessary and sufficient condition for the existence of PGFR on Cayley graphs over abelian groups. They  also generalized several pre-existing results of PGFR on circulant graphs. Wang et al.~\cite{PGFR Nonabelian} studied PGFR on normal Cayley graphs over dicyclic groups. Further research on PGFR was conducted by Chan et al.~\cite{PGFR Laplacian} and Drazen et al.~\cite{PGFR diagonal perturbation}. 

Klotz and Sander~\cite{Klotz} studied structural and spectral properties of unitary Cayley graphs. Beaudrap~\cite{Beaudrap} in the year 2010 introduced quadratic unitary Cayley graphs, which are sub graphs of unitary Cayley graphs. Quadratic unitary Cayley graphs are a circulant generalization of \textit{Paley graphs}. In fact, if $n$ is a prime with $n \equiv 1~(\mathrm{mod}~4)$, then a quadratic unitary Cayley graph on $n$ vertices is exactly a Paley graph.  Huang~\cite{Jing} determined all the eigenvalues of quadratic unitary Cayley graphs $G_n$ for $n > 1$. Bhakta and Bhattacharjya~\cite{Bhakta 1} studied periodicity and perfect state transfer on unitary Cayley graphs in discrete-time quantum walks. Additionally, Bhakta and Bhattacharjya~\cite{Bhakta 2} studied periodicity and perfect state transfer on quadratic unitary Cayley graphs in discrete-time quantum walks. This motivates us to study state transfer on unitary Cayley graphs and quadratic unitary Cayley graphs in the context of continuous-time quantum walks.  

The subsequent sections of this paper are structured as follows. In Section~\ref{sec 2}, we introduce some basic definitions and preliminary results that we use in later sections. In Section~\ref{sec 3}, we classify unitary Cayley graphs that admit PGFR and those that admit FR, in terms of the number of vertices. It turns out that a unitary Cayley graph admits FR if and only if it admits PGFR. In Section~\ref{sec 4}, we obtain all periodic quadratic unitary Cayley graphs. As a consequence, we obtain all quadratic unitary Cayley graphs admitting PST. Complete classifications of quadratic unitary Cayley graphs admitting PGST and PGFR are given in Sections~\ref{sec 5} and~\ref{sec 6}, respectively. In Section~\ref{sec 7}, we obtain all quadratic unitary Cayley graphs admitting FR. It turns out that there are infinitely many quadratic unitary Cayley graphs admitting PGFR that fail to admit PGST and FR.

\section{Basic definitions and preliminary results}\label{sec 2}
In this section, we discuss some basic definitions and results that are useful for the study of state transfer on graphs. Let $n$ be a positive integer and ${\omega}_n=\exp(\frac{2 \pi \textbf{i}}{n})$. Let ${\mathbb{Z}}_n=\{0, \ldots, n-1\}=\{a_0, \ldots, a_{n-1}\}$ such that $a_0=0$. Also, let $r$ and $a_r$ have the same parity for $1 \leq r \leq n-1$. The following theorem is useful to determine the eigenvalues and the corresponding eigenvectors of the circulant graph $\mathrm{Cay}({\mathbb{Z}}_n, S)$.
\begin{theorem}\label{spectrum circulant graph}\emph{\cite{Steinberg}}
The eigenvalues of the circulant graph $\mathrm{Cay}({\mathbb{Z}}_n, S)$ are given by $\lambda_0, \ldots, \lambda_{n-1}$, where 
$$\lambda_r=\sum_{y \in S}\cos \left(\frac{2 \pi y a_r}{n}\right)~~~\mathrm{for}~0 \leq r \leq n-1.$$
Moreover, $\textbf{v}_r$ is an eigenvector of $\mathrm{Cay}({\mathbb{Z}}_n, S)$ corresponding to the eigenvalue $\lambda_r$, where
$$\textbf{v}_r=\frac{1}{\sqrt{n}} [{\omega_n}^{(a_0-1)a_r},~~{\omega_n}^{(a_1-1)a_r},~~\cdots~~,{\omega_n}^{(a_{n-1}-1)a_r}]^T~~~\mathrm{for}~0 \leq r \leq n-1.$$
\end{theorem}
Let $A$ be the adjacency matrix of $\mathrm{Cay}({\mathbb{Z}}_n, S)$ with the spectral decomposition $A=\sum_{r = 0}^{n-1} \lambda_r E_r$, where $E_r={\textbf{v}}_r {\textbf{v}}^\star_r$. Then the spectral decomposition of the transition matrix $H(t)$ is given by
\begin{align*}
H(t)=\sum_{r=0}^{n-1} \exp(-\textbf{i}t\lambda_r) E_r.
\end{align*}
It follows that the $ab$-th entry of $H(t)$ is 
$$H(t)_{ab}=\frac{1}{n}\sum_{r=0}^{n-1}\exp({-\textbf{i}t{\lambda_r}}){\omega_n}^{(a-b)a_r}.$$ 

Tan et al.~\cite{Tan} provided a necessary condition for the existence of PST on integral abelian Cayley graphs. 
\begin{theorem}\label{Tan}\emph{\cite{Tan}}
Let $n$ be an integer such that $n > 3$. If an integral abelian Cayley graph on $n$ vertices admits PST, then $n$ must be divisible by four.
\end{theorem}

The following theorem is due to Pal and Bhattacharjya.
\begin{theorem}\label{Pal circulant cycle}\emph{\cite{Pal circulant}}
A cycle on $n$ vertices admits PGST if and only if $n=2^h$, where $h$ is an integer such that $h > 1$.
\end{theorem}

Chan et al.~\cite{Chan FR} obtained a complete classification of $C_n$ admitting FR in terms of $n$. 
\begin{theorem}\label{Chan FR}\emph{\cite{Chan FR}}
A cycle on $n$ vertices admits FR if and only if $n=4$ or $n=6$.
\end{theorem}

Wang et al.~\cite{J. Wang} proved the following theorem, which gives a necessary and sufficient condition for the existence of FR on circulant graphs.
\begin{theorem}\label{J. Wang}\emph{\cite{J. Wang}}
Let $a, b$ be two vertices of the circulant graph $\mathrm{Cay}({\mathbb{Z}}_n, S)$ such that $a < b$. Also, let $N$ denote the set $\{(x, y) \colon x, y \in {\mathbb{Z}}_n, x > y~\mathrm{and}~x-y~\mathrm{is~even}\}$. Then $\mathrm{Cay}({\mathbb{Z}}_n, S)$ admits FR from $a$ to $b$ if and only if the following two conditions hold. 
\begin{enumerate}
\item[(i)] $n$ is an even number and $b=a+\frac{n}{2}$.

\item[(ii)]  There exists a positive real number $t$ such that 
$$\frac{t}{2 \pi}(\lambda_x-\lambda_y) \in \mathbb{Z}~~~~\mathrm{for~all}~(x, y) \in N.$$ 
\end{enumerate}
\end{theorem} 
The following corollary is an immediate consequence of Theorem~\ref{J. Wang}.
\begin{corollary}\label{non existence of FR}
Let $n$ be an even positive integer such that $\mathrm{Cay}({\mathbb{Z}}_n, S)$ admits FR. Then 
$$\frac{\lambda_{x_1}-\lambda_{y_1}}{\lambda_{x_2}-\lambda_{y_2}} \in \mathbb{Q},$$ 
for all $(x_1, y_1), (x_2, y_2) \in N$ such that $\lambda_{x_2} \neq \lambda_{y_2}$.   
\end{corollary} 

The authors in~\cite{Kalita} obtained a necessary and sufficient condition for the existence of PGFR on circulant graphs.
\begin{theorem}\label{PGFR circulant}\emph{\cite{Kalita}}
Let $a, b$ be two vertices of the circulant graph $\mathrm{Cay}({\mathbb{Z}}_n, S)$ such that $a < b$. Then $\mathrm{Cay}({\mathbb{Z}}_n, S)$ admits PGFR from $a$ to $b$ if and only if the following two conditions hold. 
\begin{enumerate}
\item[(i)] $n$ is an even number and $b=a+\frac{n}{2}$.

\item[(ii)] For arbitrary integers $\ell_1, \ldots, \ell_{n-1}$, the relation 
$$\sum_{r=1}^{n-1}{\ell_r}(\lambda_r-\lambda_0) = 0$$
implies 
$$\sum_{r ~ \mathrm{odd}~}\ell_{r} \neq \pm 1.$$
\end{enumerate}
\end{theorem}
Theorem~\ref{PGFR circulant} is our main tool to investigate the existence of PGFR on the unitary Cayley graph $X_n$ and the quadratic unitary Cayley graph $G_n$.

Let $a$ be an integer and $p$ be an odd prime such that $\gcd(a,p) = 1$. Then $a$ is called a \textit{quadratic residue} (resp. \textit{quadratic non-residue}) of $p$ if the congruence $x^2 \equiv a~(\mathrm{mod}~p)$ has a solution (resp. no solution). The \textit{Legendre symbol}, denoted $(a/p)$, is defined as
\[(a/p)= \left\{ \begin{array}{rl} 1 & \textrm{ if }~a~\mathrm{is~a~quadratic~residue~of}~p\\
 -1 & \textrm{ if }~a~\textrm{is a quadratic non-residue of}~p.\\
\end{array}\right. \]
Now we state some useful number-theoretic results.
\begin{lema}\emph{\cite{Burton}}
If $p$ is an odd prime, then there are precisely $\frac{p-1}{2}$ quadratic residues and $\frac{p-1}{2}$ quadratic non-residues of $p$.
\end{lema}

\begin{lema}\emph{\cite{Burton}}
Let $p$ be an odd prime. Then the quadratic congruence $x^2 \equiv -1~(\mathrm{mod}~p)$ has a solution if and only if $p \equiv 1~(\mathrm{mod}~4)$.
\end{lema}

\begin{lema}\emph{\cite{Burton}}
Let $p$ be an odd prime and $a$ be an integer with $\gcd(a,p) = 1$. The quadratic congruence $x^2 \equiv a~(\mathrm{mod}~p^s)~\mathrm{for}~s \in \mathbb{N}$ 
has a solution if and only if $(a/p) = 1$.    
\end{lema}

\begin{lema}\label{lemma 1.2}\emph{\cite{Jing}}
Let $k$ be a positive integer such that $k > 1$. Then $-1 \notin Q_{2^k}$.
\end{lema}

For a prime $p$ and a positive integer $k$, let $\chi_a$ be a character of ${\mathbb{Z}}_{p^k}$ indexed by $a \in {\mathbb{Z}}_{p^k}$. The \textit{character sum}, denoted $\chi_a(Q_{p^k})$, is defined as 
$$\chi_a(Q_{p^k}) = \sum_{b \in Q_{p^k}} \chi_a(b).$$
Huang~\cite{Jing} provided explicit formulas to compute $\chi_a(Q_{p^k})$, which we include in the form of the next theorem.
\begin{theorem}\label{theorem character prime power}\emph{\cite{Jing}}
Let $p$ be a prime, $k$ be a positive integer and $a \in {\mathbb{Z}}_{p^k}$.
\begin{enumerate}
\item [(i)] If $p=2$, then $\chi_a(Q_2) = \cos(a \pi)$ and $\chi_a(Q_4) = \exp(-\frac{a \pi \textbf{i}}{2})$,  

and for $k > 2$,
\[\chi_a(Q_{2^k})=\left\{ \begin{array}{ll} 
2^{k-3}\exp \left(-\frac{a \pi \textbf{i}}{2^{k-1}} \right) & \textrm{ if } a \in \{2^{k-3}, 3.2^{k-3},5.2^{k-3},7.2^{k-3}\}\\
2^{k-3}(-\textbf{i})^{\frac{a}{2^{k-2}}}& \textrm{ if } a \in \{0, 2^{k-2}, 2^{k-1}, 3.2^{k-2}\}\\
0 & \textrm{ otherwise. } \end{array}\right. \]

\item [(ii)] If $p$ is odd, then for all $k$ 
\[\chi_a(Q_{p^k})=\left\{ \begin{array}{ll} 
\frac{p^{k-1}(p-1)}{2} & \textrm{ if } a=0\\
\frac{1}{2}p^{k-1}\left[\sqrt{p} \left(\frac{a}{p^{k-1}}/p \right)\frac{1+{\textbf{i}}^p}{1+\textbf{i}}-1 \right]   & \textrm{ if } a \in p^{k-1}{\mathbb{Z}}_p \setminus \{0\}\\
0 & \textrm{ otherwise. } \end{array}\right. \]
\end{enumerate}
\end{theorem}

Using the next theorem, one can determine the eigenvalues of quadratic unitary Cayley graphs $G_n$ for $n > 1$. 
\begin{theorem}\label{theorem canonical factorization}\emph{\cite{Jing}}
Let $n$ be a positive integer strictly greater than one such that the canonical factorization of $n$ is given by $n={p_1}^{k_1} \cdots {p_\ell}^{k_\ell}$. Also, let $0 \leq t \leq \ell$ with $-1 \in Q_{{p_j}^{k_j}}$ for $1 \leq j \leq t$ and $-1 \notin Q_{{p_j}^{k_j}}$ for $t+1 \leq j \leq \ell$.
\begin{enumerate}
\item [(i)] If $ t=\ell$, then the eigenvalues of $G_n$ are given by
$$\chi_{a_1} \left(Q_{{p_1}^{k_1}} \right) \cdots \chi_{a_{\ell}} \left(Q_{{p_{\ell}}^{k_{\ell}}} \right),~~~\mathrm{where}~a_j \in {\mathbb{Z}}_{{p_j}^{k_j}}~\mathrm{for}~1 \leq j \leq \ell.$$ 

\item [(ii)] If $ t < \ell$, then the eigenvalues of $G_n$ are given by
$$\chi_{a_1} \left(Q_{{p_1}^{k_1}} \right) \cdots \chi_{a_{t}} \left(Q_{{p_{t}}^{k_{t}}} \right)\left[\chi_{a_{t+1}} \left(Q_{{p_{t+1}}^{k_{t+1}}} \right) \cdots \chi_{a_{\ell}}\left(Q_{{p_{\ell}}^{k_{\ell}}} \right) + \overline {\chi_{a_{t+1}}\left(Q_{{p_{t+1}}^{k_{t+1}}} \right)} \cdots \overline{\chi_{a_{\ell}}\left(Q_{{p_{\ell}}^{k_{\ell}}} \right)}\right],$$ 
$\mathrm{where}~a_j \in {\mathbb{Z}}_{{p_j}^{k_j}}~\mathrm{for}~1 \leq j \leq \ell$.
\end{enumerate}
\end{theorem}

\section{State transfer on unitary Cayley graphs}\label{sec 3}
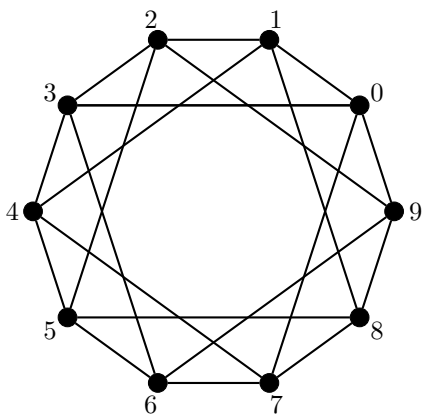
\begin{figure}
\begin{center}
	\begin{tikzpicture}[scale=1.2,auto,swap]
    \tikzstyle{blackvertex}=[circle,draw=black,fill=black]
    \tikzstyle{bluevertex}=[circle,draw=blue,fill=blue]
    \tikzstyle{greenvertex}=[circle,draw=green,fill=green]
    \tikzstyle{yellowvertex}=[circle,draw=yellow,fill=yellow]
    \tikzstyle{brownvertex}=[circle,draw=brown,fill=brown]
    \tikzstyle{redvertex}=[circle,draw=red,fill=red]
    \tikzstyle{blackvertex}=[circle,draw=black,fill=black]
    \tikzstyle{bluevertex}=[circle,draw=blue,fill=blue]
    \tikzstyle{greenvertex}=[circle,draw=green,fill=green]
    \tikzstyle{yellowvertex}=[circle,draw=yellow,fill=yellow]
    \tikzstyle{brownvertex}=[circle,draw=brown,fill=brown]
    \tikzstyle{redvertex}=[circle,draw=red,fill=red]
    \tikzstyle{brownvertex}=[circle,draw=brown,fill=brown]
    \tikzstyle{redvertex}=[circle,draw=red,fill=red]

    \node [blackvertex,scale=0.75] (a0) at (36:2) {};
    \node [blackvertex,scale=0.75] (a1) at (72:2) {};
    \node [blackvertex,scale=0.75] (a2) at (108:2) {};
    \node [blackvertex,scale=0.75] (a3) at (144:2) {};
    \node [blackvertex,scale=0.75] (a4) at (180:2) {};
    \node [blackvertex,scale=0.75] (a5) at (216:2) {};
    \node [blackvertex,scale=0.75] (a6) at (252:2) {};
    \node [blackvertex,scale=0.75] (a7) at (288:2) {};
    \node [blackvertex,scale=0.75] (a8) at (324:2) {};
    \node [blackvertex,scale=0.75] (a9) at (360:2) {};

    \node [scale=1] at (36:2.24) {0};
    \node [scale=1] at (72:2.24) {1};
    \node [scale=1] at (108:2.24) {2};
    \node [scale=1] at (144:2.24) {3};
    \node [scale=1] at (180:2.23) {4};
    \node [scale=1] at (216:2.24) {5};
    \node [scale=1] at (252:2.25) {6};
    \node [scale=1] at (288:2.25) {7};
    \node [scale=1] at (324:2.24) {8};
    \node [scale=1] at (360:2.24) {9};
    
\draw [black,thick] (a0) -- (a1) -- (a2) -- (a3) -- (a4) -- (a5) -- (a6) --(a7) -- (a8) -- (a9) -- (a0);
\draw [black,thick] (a0) -- (a3);
\draw [black,thick] (a0) -- (a7);
\draw [black,thick] (a1) -- (a4);
\draw [black,thick] (a1) -- (a8);
\draw [black,thick] (a2) -- (a5);
\draw [black,thick] (a2) -- (a9);
\draw [black,thick] (a3) -- (a0);
\draw [black,thick] (a3) -- (a6);
\draw [black,thick] (a4) -- (a7);
\draw [black,thick] (a5) -- (a8);
\draw [black,thick] (a6) -- (a9);
\end{tikzpicture}
\caption{Unitary Cayley graph on $10$ vertices admitting PGFR}
\label{figure1}
\end{center}
\end{figure}

Basic et al.~\cite{Basic} proved that the unitary Cayley graph $X_n$ admits PST if and only if $n=2$ or $n=4$. Since $X_n$ is well known to admit periodicity, it follows from Pal~\cite{Pal PhD Thesis} [see Proposition 1.4] and also from Pal~\cite{State transfer on circulant Pal} [see Proposition 1.3.1] that $X_n$ admits PGST if and only if $n=2$ or $n=4$. In this section, we explore FR and PGFR on $X_n$. We classify unitary Cayley graphs $X_n$ admitting PGFR. We also classify $X_n$ admitting FR. It turns out that $X_n$ admits FR if and only if it admits PGFR. The following lemma is due to Klotz and Sander~\cite{Klotz}, which is useful to determine the eigenvalues of $X_n$.

\begin{lema}\emph{\cite{Klotz}}\label{unitary Cayley graph 1}
Let $\mu$ denote the $M\ddot{o}bius$ function and $\varphi$ denote Euler's phi function. Also, let $n$ be a positive integer and $r$ be an integer such that $0 \leq r \leq n-1$. Then the eigenvalues of $X_n$ are given by
\begin{align*}
\lambda_r = \mu(c(r, n)) \frac{\varphi(n)}{\varphi(c(r, n))},~~~\mathrm{where}~c(r, n) = \frac{n}{\gcd(r, n)}.
\end{align*}
\end{lema}
The authors in~\cite{Kalita} proved the following theorem, which gives a sufficient condition for the non-existence of PGFR on circulant graphs. 
\begin{theorem}\label{PGFR non existence circulant}\emph{\cite{Kalita}}
Let $n$ be a positive integer such that $2pq \mid n$, where $p$ and $q$ are two distinct odd primes. Also, let $S$ be a connection set in ${\mathbb{Z}}_n$ such that $p \nmid y$ and $q \nmid y$ for all $y \in S$. Then $\mathrm{Cay}({\mathbb{Z}}_n, S)$ does not admit PGFR.
\end{theorem}
Now we use Theorem~\ref{PGFR non existence circulant} to prove the following lemma.

\begin{lema}\label{unitary Cayley graph 2}
The graph $X_n$ does not admit PGFR if $2pq \mid n$, for some distinct odd primes $p$ and $q$.
\end{lema}
\begin{proof}
Recall that $X_n=\mathrm{Cay}({\mathbb{Z}}_n, U(n))$. If $u \in U(n)$, then $p \nmid u$ and $q \nmid u$. Therefore, Theorem~\ref{PGFR non existence circulant} yields that the graph $X_n$ does not admit PGFR. 
\end{proof}
From Theorem~\ref{PGFR circulant} and Lemma~\ref{unitary Cayley graph 2}, we observe that if $X_n$ admits PGFR, then it is necessary that $n=2^h p^s$, where $h \in \mathbb{N}$, $s \in \mathbb{N} \cup \{0\}$ and $p$ is an odd prime. It is obvious that $X_n$ admits PGFR for $n=2$ and $n=4$. Now we assume that $n \geq 6$. The following lemma provides an infinite family of $X_n$ admitting PGFR.

\begin{lema}\label{unitary Cayley graph 3}
Let $n=2p$, where $p$ is an odd prime. Then $X_n$ admits PGFR.
\end{lema}
\begin{proof}
The eigenvalues of $X_n$ are given by
\[{\lambda}_r=\left\{ \begin{array}{ll} 
p-1 & \textrm{ if } r=0\\
-1 & \textrm{ if } r~\mathrm{is~even~and}~r \neq 0\\
 1-p & \textrm{ if } r=p~\\
 1 & \textrm{ otherwise. } \end{array}\right. \]
Suppose the integers $\ell_1,\ldots,\ell_{n-1}$ satisfy the relation 
\begin{align*}
\sum_{r=1}^{n-1}{\ell_r}(\lambda_r-\lambda_0)=0.
\end{align*}
Then we have  
\begin{align*}
(2-p)\sum_{r~\mathrm{odd}}\ell_r=p(\ell_p+\sum_{r~\mathrm{even}}\ell_r).
\end{align*}
This implies that $p \mid \sum_{r~\mathrm{odd}}\ell_r$ and therefore $\sum_{r~\mathrm{odd}}\ell_r \neq \pm 1$. Now Theorem~\ref{PGFR circulant} yields that $X_n$ admits PGFR. This completes the proof.
\end{proof}
The unitary Cayley graph on 10 vertices admitting PGFR is shown in Figure~\ref{figure1}. In the next lemma, we prove that $X_{4p}$ does not admit PGFR, where $p$ is an odd prime.

\begin{lema}\label{unitary Cayley graph 4}
Let  $n=4p$, where $p$ is an odd prime. Then $X_n$ does not admit PGFR.
\end{lema}
\begin{proof}
We have $\lambda_1=0,~\lambda_2=2$ and $\lambda_8=-2$.
Define the integers $\ell_1, \ldots, \ell_{n-1}$ by
\[{\ell}_r=\left\{ \begin{array}{ll} 1 & \textrm{ if } r=1\\
 \frac{p-1}{2} & \textrm{ if } r=2~\\
 \frac{1-p}{2} & \textrm{ if } r=8~\\
 0 & \textrm{ otherwise. } \end{array}\right. \]
Then 
$\sum_{r=1}^{n-1}{\ell_r}(\lambda_r-\lambda_0)=0$ and $\sum_{r~\mathrm{odd}} \ell_r=1$. Now Theorem~\ref{PGFR circulant} yields that $X_n$ does not admit PGFR.
\end{proof}

\begin{lema}\label{unitary Cayley graph 5}
Let $n=2^h p^s$, where $h, s$ are integers such that $h > 2$ and $s \in \{0, 1\}$. Then $X_n$ does not admit PGFR.
\end{lema}
\begin{proof}
We have $\lambda_{1}=\lambda_{2}=0$. Define the integers $\ell_1, \ldots, \ell_{n-1}$ by
\[{\ell}_r=\left\{ \begin{array}{rl} 1 & \textrm{ if } r=1\\
 -1 & \textrm{ if } r=2~\\
 0 & \textrm{ otherwise. } \end{array}\right. \]
Then 
$\sum_{r=1}^{n-1}{\ell_r}(\lambda_r-\lambda_0)=0$ and $\sum_{r~\mathrm{odd}} \ell_r=1$. Now Theorem~\ref{PGFR circulant} yields that $X_n$ does not admit PGFR. 
\end{proof}

\begin{lema}\label{unitary Cayley graph 6}
Let $n=2^h p^s$, where $h \in \mathbb{N}$ and $s$ is an integer such that $s > 1$. Then the graph $X_n$ does not admit PGFR.
\end{lema}
\begin{proof}
We have $\lambda_{1}=\lambda_{2^h}=0$. Now the rest of the proof is similar to that of Lemma~\ref{unitary Cayley graph 5}, and hence the details are omitted.
\end{proof}

We combine the preceding lemmas to present the following theorem.
\begin{theorem}\label{unitary Cayley graph 7}
The unitary Cayley graph $X_n$ admits PGFR if and only if $n=2$ or $n=2p$, where $p$ is a prime.
\end{theorem}
Now we characterize the existence of FR on the unitary Cayley graph $X_n$. From Theorem~\ref{unitary Cayley graph 7}, it is obvious that if $X_n$ admits FR, then either $n=2$ or $n=2p$ for some prime $p$. The graphs $X_2$ and $X_4$ admit FR. In the following theorem, we use Theorem~\ref{J. Wang} to prove that $X_n$ admits FR, where $n=2p$ and $p$ is an odd prime.

\begin{theorem}\label{UCG FR}
Let $n=2p$, where $p$ is an odd prime. Then the graph $X_n$ admits FR.
\end{theorem}
\begin{proof}
From Lemma~\ref{unitary Cayley graph 3}, recall that the eigenvalues of $X_n$ are given by
\[{\lambda}_r=\left\{ \begin{array}{ll} 
p-1 & \textrm{ if } r=0\\
-1 & \textrm{ if } r~\mathrm{is~even~and}~r \neq 0\\
 1-p & \textrm{ if } r=p~\\
 1 & \textrm{ otherwise. } \end{array}\right. \]
Let $(x, y) \in N$. Then the following five cases arise.

\noindent \textbf{Case 1}. $x$ is even and $y=0$. In this case, $\lambda_x-\lambda_y=-p$. 
 
\noindent \textbf{Case 2}. $x$ and $y$ both are even, and non-zero. In this case, $\lambda_x-\lambda_y=0$. 
 
\noindent \textbf{Case 3}. $x=p$ and $y$ is odd. In this case, $\lambda_x-\lambda_y =-p$. 
 
\noindent \textbf{Case 4}. $x$ and $y$ both are odd, and different from $p$. In this case, $\lambda_x-\lambda_y=0$.  
 
\noindent \textbf{Case 5}. $x$ is odd and $y=p$. In this case, $\lambda_x-\lambda_y=p$.  

Thus $\lambda_x-\lambda_y \in \{-p, 0 , p\}$ for all $(x, y) \in N$. For $t=\frac{2 \pi}{p}$, we have $\frac{t}{2 \pi}(\lambda_x-\lambda_y) \in \{-1, 0 , 1\}$ for all $(x, y) \in N$. Now Theorem~\ref{J. Wang} yields that $X_n$ admits FR. This completes the proof.
\end{proof}
Thus, we find that the unitary Cayley graph $X_n$ admits FR if and only if $n=2$ or $n=2p$ for some prime $p$. The following theorem is the main result of this section.

\begin{theorem}\label{main result unitary Cayley graph}
Let $n$ be a positive integer. Then the unitary Cayley graph $X_n$ admits FR if and only if it admits PGFR.
\end{theorem}

\section{Periodicity and PST on quadratic unitary Cayley graphs}\label{sec 4}
In this section, we first determine the periodic quadratic unitary Cayley graphs. As a consequence, we obtain quadratic unitary Cayley graphs admitting PST. It follows from Godsil~\cite{State transfer on graphs} [see Theorem 5.2] that a quadratic unitary Cayley graph $G_n$ is periodic if and only if it is integral. In the next few lemmas, we determine all $G_n$ that are integral.  

\begin{lema}\label{2^h}
Let $h$ be a positive integer. Then the graph $G_{2^h}$ is integral if and only if $h = 1$ or $h = 2$.
\end{lema}
\begin{proof}
The graph $G_2$ is the path on two vertices and the graph $G_4$ is the cycle on four vertices. Therefore, $G_2$ and $G_4$ are integral. Now consider the case that $h > 2$. We consider $a \in {\mathbb{Z}}_{2^h}$ such that $a=2^{h-3}$. Applying Lemma~\ref{lemma 1.2}, Theorem~\ref{theorem character prime power} and Theorem~\ref{theorem canonical factorization}, we obtain that
$$\lambda_a=\chi_a(Q_{2^h})+\overline{\chi_a(Q_{2^h})}=\frac{2^h}{4 \sqrt{2}}.$$
Therefore, the graph $G_{2^h}$ is not integral. This completes the proof. 
\end{proof}

\begin{lema}\label{p^s}
Let $p$ be a prime such that $p \equiv 3~(\mathrm{mod}~4)$ and $s$ be a positive integer. Then $G_{p^s}$ is integral.
\end{lema}
\begin{proof}
We have $\lambda_a=\chi_a(Q_{p^s})+\overline{\chi_a(Q_{p^s})}$ for $a \in {\mathbb{Z}}_{p^s}$. This implies that
\[{\lambda}_a=\left\{ \begin{array}{ll} p^{s-1}(p-1) & \textrm{ if } a=0\\
 -p^{s-1} & \textrm{ if } a \in p^{s-1} {\mathbb{Z}}_p \setminus \{0\}~\\
 0 & \textrm{ otherwise. } \end{array}\right. \]
Thus it is clear that the graph $G_{p^s}$ is integral.
\end{proof}

The proof of the next lemma is similar to that of Lemma~\ref{p^s}, and hence the details are omitted.
\begin{lema}\label{2p^s}
Let $p$ be a prime such that $p \equiv 3~(\mathrm{mod}~4)$ and $s$ be a positive integer. Then $G_{2p^s}$ is integral.
\end{lema}

\begin{lema}\label{odd, only prime 4k+3}
Let $k$ be an integer such that $k > 1$. Also, let $p_j$ be a prime with $p_j \equiv 3~(\mathrm{mod}~4)$ and $s_j$ be a positive integer for $1 \leq j \leq k$. If $n={p_1}^{s_1} \cdots {p_k}^{s_k}$, then $G_n$ is non-integral.
\end{lema}
\begin{proof}
We consider the following two cases.

\noindent \textbf{Case 1.} For $k=2$, we write ${\mathbb{Z}}_n={\mathbb{Z}}_{{p_1}^{s_1}} \times {\mathbb{Z}}_{{p_2}^{s_2}}$. Then we have  
$$\lambda_{(a_1, a_2)}=\chi_{a_1} \left(Q_{{p_1}^{s_1}} \right) \chi_{a_2}\left(Q_{{p_2}^{s_2}} \right)+\overline {\chi_{a_1}\left(Q_{{p_1}^{s_1}} \right)} \overline{\chi_{a_2}\left(Q_{{p_2}^{s_2}} \right)},$$
where $a_1 \in {\mathbb{Z}}_{{p_1}^{s_1}}$ and $a_2 \in {\mathbb{Z}}_{{p_2}^{s_2}}$. Consider $a_1 \in {\mathbb{Z}}_{{p_1}^{s_1}} \setminus \{0\}$ with $\left(\frac{a_1}{{p_1}^{s_1-1}}/p_1 \right)=1$ and $a_2 \in {\mathbb{Z}}_{{p_2}^{s_2}} \setminus \{0\}$ with $\left(\frac{a_2}{{p_2}^{s_2-1}}/p_2 \right)=-1$. Then    
$$\lambda_{(a_1, a_2)}=\theta_1+\theta_1 \sqrt{p_1 p_2},~~~~\mathrm{for~some}~\theta_1 \in \mathbb{Q} \setminus \{0\}.$$

\noindent \textbf{Case 2.} For $k > 2$, we write ${\mathbb{Z}}_n={\mathbb{Z}}_{{p_1}^{s_1}} \times \cdots \times {\mathbb{Z}}_{{p_k}^{s_k}}$. Then we have  
$$\lambda_{(a_1, \ldots, a_k)}=\chi_{a_1} \left(Q_{{p_1}^{s_1}} \right) \cdots \chi_{a_k}\left(Q_{{p_k}^{s_k}} \right)+\overline {\chi_{a_1}\left(Q_{{p_1}^{s_1}} \right)} \cdots \overline{\chi_{a_k}\left(Q_{{p_k}^{s_k}} \right)},$$
where $a_j \in {\mathbb{Z}}_{{p_j}^{s_j}}$ for $1 \leq j \leq k$. Consider $a_1 \in {\mathbb{Z}}_{{p_1}^{s_1}} \setminus \{0\}$ such that $\left(\frac{a_1}{{p_1}^{s_1-1}}/p_1 \right)=1$, $a_2 \in {\mathbb{Z}}_{{p_2}^{s_2}} \setminus \{0\}$ such that $\left(\frac{a_2}{{p_2}^{s_2-1}}/p_2 \right)=-1$ and $a_3=\cdots=a_k=0$. Then 
$$\lambda_{(a_1, \ldots, a_k)}=\theta_2+\theta_2 \sqrt{p_1 p_2},~~~~\mathrm{for~some}~\theta_2 \in \mathbb{Q} \setminus \{0\}.$$ 
Therefore, the graph $G_n$ is non-integral.
\end{proof}

The proof of the next lemma follows similarly to that of Lemma~\ref{odd, only prime 4k+3}, and hence the details are omitted.
\begin{lema}
Let $k$ be an integer such that $k > 1$. Also, let $p_j$ be a prime with $p_j \equiv 3~(\mathrm{mod}~4)$ and $s_j$ be a positive integer for $1 \leq j \leq k$. If $n=2{p_1}^{s_1} \cdots {p_k}^{s_k}$, then $G_n$ is non-integral.
\end{lema}

\begin{lema}\label{odd, only prime 4k+1}
Let $k$ be a positive integer. Also, let $p_j$ be a prime such that $p_j \equiv 1~(\mathrm{mod}~4)$ and $s_j$ be a positive integer for $1 \leq j \leq k$. If $n={p_1}^{s_1} \cdots {p_k}^{s_k}$, then $G_n$ is non-integral.
\end{lema}
\begin{proof}
We consider the following two cases.

\noindent \textbf{Case 1.} For $k=1$, we have $\lambda_{a_1}=\chi_{a_1}(Q_{{p_1}^{s_1}})$, where $a_1 \in {\mathbb{Z}}_{{p_1}^{s_1}}$. Consider $a_1 \in {\mathbb{Z}}_{{p_1}^{s_1}} \setminus \{0\}$ such that $\left(\frac{a_1}{{p_1}^{s_1-1}}/p_1 \right)=1$. Then 
$$ \lambda_{a_1}=\frac{{p_1}^{s_1}}{2} \left (-1+\sqrt{p_1} \right).$$

\noindent \textbf{Case 2.} For $k > 1$, we write ${\mathbb{Z}}_n={\mathbb{Z}}_{{p_1}^{s_1}} \times \cdots \times {\mathbb{Z}}_{{p_k}^{s_k}}$. Then we have 
$$\lambda_{(a_1, \ldots, a_k)}=\chi_{a_1}(Q_{{p_1}^{s_1}})\cdots \chi_{a_k}(Q_{{p_k}^{s_k}}),$$
where $a_j \in {\mathbb{Z}}_{{p_j}^{s_j}}$ for $1 \leq j \leq k$. Consider $a_1 \in {\mathbb{Z}}_{{p_1}^{s_1}} \setminus \{0\}$ such that $\left(\frac{a_1}{{p_1}^{s_1-1}}/p_1 \right)=1$ and $a_j=0$ for $2 \leq j \leq k$. Then
$$\lambda_{(a_1, \ldots, a_k)}=-\theta+\theta \sqrt{p_1}~~~\mathrm{for~some}~\theta \in \mathbb{Q} \setminus \{0\}.$$
Thus, the graph $G_n$ is non-integral. 
\end{proof}
The proof of the following lemma is similar to that of Lemma~\ref{odd, only prime 4k+1}.
\begin{lema}
Let $k$ be a positive integer. Also, let $p_j$ be a prime such that $p_j \equiv 1~(\mathrm{mod}~4)$ and $s_j$ be a positive integer for $1 \leq j \leq k$. If $n=2{p_1}^{s_1} \cdots {p_k}^{s_k}$, then $G_n$ is non-integral.
\end{lema}

\begin{lema}\label{odd, prime 4k+1 and prime 4k+3}
Let $k$ and $t$ be positive integers such that $1 \leq t < k$. Let $p_j$ be a prime with $p_j \equiv 1~(\mathrm{mod}~4)$ for $1 \leq j \leq t$, $p_j \equiv 3~(\mathrm{mod}~4)$ for $t+1 \leq j \leq k$ and $s_j$ be a positive integer for $1 \leq j \leq k$. If $n={p_1}^{s_1} \cdots {p_t}^{s_t} {p_{t+1}}^{s_{t+1}} \cdots {p_k}^{s_k}$, then $G_n$ is non-integral.
\end{lema}
\begin{proof}
We write ${\mathbb{Z}}_n={\mathbb{Z}}_{{p_1}^{s_1}} \times \cdots \times {\mathbb{Z}}_{{p_k}^{s_k}}$. Then we have  

$$\resizebox{0.999\hsize}{!}{$\lambda_{(a_1, \ldots, a_k)}=\chi_{a_1} \left(Q_{{p_1}^{s_1}} \right) \cdots \chi_{a_{t}} \left(Q_{{p_{t}}^{s_{t}}} \right)\left[\chi_{a_{t+1}} \left(Q_{{p_{t+1}}^{s_{t+1}}} \right) \cdots \chi_{a_{\ell}}\left(Q_{{p_k}^{s_k}} \right)+\overline {\chi_{a_{t+1}}\left(Q_{{p_{t+1}}^{s_{t+1}}} \right) \cdots \chi_{a_k}\left(Q_{{p_k}^{s_k}} \right)}\right],$}$$
where $a_j \in {\mathbb{Z}}_{{p_j}^{s_j}}$ for $1 \leq j \leq k$. Consider $a_1 \in {\mathbb{Z}}_{{p_1}^{s_1}} \setminus \{0\}$ such that $\left(\frac{a_1}{{p_1}^{s_1-1}}/p_1 \right)=1$ and $a_j=0$ for $2 \leq j \leq k$. Then 
$$\lambda_{(a_1, \ldots, a_k)}=-\theta+\theta \sqrt{p_1}~~~\mathrm{for~some}~\theta \in \mathbb{Q} \setminus \{0\}.$$ 
Therefore, the graph $G_n$ is non-integral.
\end{proof}

The proof of the following lemma is similar to that of Lemma~\ref{odd, prime 4k+1 and prime 4k+3}.
\begin{lema}
Let $k$ and $t$ be positive integers such that $1 \leq t < k$. Let $p_j$ be a prime with $p_j \equiv 1~(\mathrm{mod}~4)$ for $1 \leq j \leq t$, $p_j \equiv 3~(\mathrm{mod}~4)$ for $t+1 \leq j \leq k$ and $s_j$ be a positive integer for $1 \leq j \leq k$. If $n=2{p_1}^{s_1} \cdots {p_t}^{s_t} {p_{t+1}}^{s_{t+1}} \cdots {p_k}^{s_k}$, then $G_n$ is non-integral.
\end{lema}

\begin{lema}
Let $h$ be an integer with $h > 1$, $p$ be a prime with $p \equiv 3~(\mathrm{mod}~4)$ and $s$ be a positive integer. If $n=2^h p^s$, then $G_n$ is non-integral.
\end{lema} 
\begin{proof}
We write ${\mathbb{Z}}_n={\mathbb{Z}}_{2^h} \times {\mathbb{Z}}_{p^s}$. Then we have  
$$\lambda_{(a_1, a_2)}=\chi_{a_1}(Q_{2^h}) \chi_{a_2} \left(Q_{p^s} \right)+\overline{\chi_{a_1}(Q_{2^h}) \chi_{a_2} \left(Q_{p^s} \right)},$$
where $a_1 \in {\mathbb{Z}}_{2^h}$ and $a_2 \in {\mathbb{Z}}_{p^s}$. Consider $a_1=2^{h-2}$ and $a_2 \in {\mathbb{Z}}_{p^s} \setminus \{0\}$ with $\left(\frac{a_2}{p^{s-1}}/p \right)=1$. 
Then  
$$\lambda_{(a_1, a_2)}=\theta \sqrt{p}~~~\mathrm{for~some}~\theta \in \mathbb{Q} \setminus \{0\}.$$
Therefore, the graph $G_n$ is non-integral.
\end{proof}

\begin{lema}\label{even, only prime 4k+3}
Let $k,h$ be integers such that $k>1$ and $h > 1$. Also, let $p_j$ be a prime with $p_j \equiv 3~(\mathrm{mod}~4)$ and $s_j$ be a positive integer for $1 \leq j \leq k$. If $n = 2^h{p_1}^{s_1} \cdots {p_k}^{s_k}$, then $G_n$ is non-integral.
\end{lema}  
\begin{proof}
We write ${\mathbb{Z}}_n={\mathbb{Z}}_{2^h} \times {\mathbb{Z}}_{{p_1}^{s_1}} \times \cdots \times {\mathbb{Z}}_{{p_k}^{s_k}}$. Then we have  
$$\lambda_{(a, a_1, \ldots, a_k)}=\chi_a(Q_{2^h}) \chi_{a_1} \left(Q_{{p_1}^{s_1}} \right) \cdots \chi_{a_k}\left(Q_{{p_k}^{s_k}} \right)+\overline{\chi_a(Q_{2^h}) \chi_{a_1} \left(Q_{{p_1}^{s_1}} \right) \cdots \chi_{a_k}\left(Q_{{p_k}^{s_k}} \right)},$$
where $a \in {\mathbb{Z}}_{2^h}$ and $a_j \in {\mathbb{Z}}_{{p_j}^{s_j}}$ for $1 \leq j \leq k$. Let $a=2^{h-2}$, $a_1 \in {\mathbb{Z}}_{{p_1}^{s_1}} \setminus \{0\}$ such that $\left(\frac{a_1}{{p_1}^{s_1-1}}/p_1 \right)=1$ and $a_j=0$ for $2 \leq j \leq k$. 
Then 
$$\lambda_{(a, a_1, \ldots, a_k)}=\theta \sqrt{p}~~~\mathrm{for~some}~\theta \in \mathbb{Q} \setminus \{0\}.$$
Therefore, the graph $G_n$ is non-integral. 
\end{proof}

\begin{lema}
Let $h$ be an integer with $h > 1$, $p$ be a prime with $p \equiv 1~(\mathrm{mod}~4)$ and $s$ be a positive integer. If $n=2^h p^s$, then $G_n$ is non-integral.
\end{lema}  
\begin{proof}
We write ${\mathbb{Z}}_n={\mathbb{Z}}_{p^s} \times {\mathbb{Z}}_{2^h}$. Then we have 
$$\lambda_{(a_1, a_2)}=\chi_{a_1}(Q_{p^s}) \left[\chi_{a_2}(Q_{2^h})+\overline{\chi_{a_2}(Q_{2^h})} \right],$$
where $a_1 \in {\mathbb{Z}}_{p^s}$ and $a_2 \in {\mathbb{Z}}_{2^h}$. Consider $a_1 \in {\mathbb{Z}}_{p^s} \setminus \{0\}$ such that $\left(\frac{a_1}{p^{s-1}}/p \right)=1$ and $a_2=0$. Then 
$$\lambda_{(a_1, a_2)}=-\theta+\theta \sqrt{p}~~~\mathrm{for~some}~\theta \in \mathbb{Q} \setminus \{0\}.$$ 
Therefore, $G_n$ is non-integral. 
\end{proof}

\begin{lema}\label{even, only prime 4k+1}
Let $h, k$ be integers with $h > 1$ and $k > 1$. Also, let $p_j$ be a prime such that $p_j \equiv 1~(\mathrm{mod}~4)$ and $s_j$ is a positive integer for $1 \leq j \leq k$. If $n=2^h{p_1}^{s_1} \cdots {p_k}^{s_k}$, then $G_n$ is non-integral.
\end{lema}  
\begin{proof}
We write ${\mathbb{Z}}_n={\mathbb{Z}}_{{p_1}^{s_1}} \times \cdots \times {\mathbb{Z}}_{{p_k}^{s_k}} \times {\mathbb{Z}}_{2^h}$. Then we have 
$$\lambda_{(a_1, \ldots, a_k, a_{k+1})}=\chi_{a_1}(Q_{{p_1}^{s_1}})\cdots \chi_{a_k}(Q_{{p_k}^{s_k}}) \left[\chi_{a_{k+1}}(Q_{2^h})+\overline{\chi_{a_{k+1}}(Q_{2^h})} \right],$$
where $a_j \in {\mathbb{Z}}_{{p_j}^{s_j}}$ for $1 \leq j \leq k$ and $a_{k+1} \in {\mathbb{Z}}_{2^h}$. Consider $a_1 \in {\mathbb{Z}}_{{p_1}^{s_1}} \setminus \{0\}$ such that $\left(\frac{a_1}{{p_1}^{s_1-1}}/p_1 \right)=1$ and $a_j=0$ for $2 \leq j \leq k+1$. Then 
$$\lambda_{(a_1, \ldots, a_k, a_{k+1})}=-\theta+\theta \sqrt{p_1}~~~\mathrm{for~some}~ \theta \in \mathbb{Q} \setminus \{0\}.$$
Therefore, $G_n$ is non-integral.
\end{proof}

The proof of the following lemma is similar to that of Lemma~\ref{even, only prime 4k+1}, and hence the details are omitted. 
\begin{lema}\label{even, prime 4k+1 and prime 4k+3}
Let $h,t,k$ be integers with $h > 1$ and $1 \leq t <k$. Also, let $p_j$ be a prime with $p_j \equiv 1~(\mathrm{mod}~4)$ for $1 \leq j \leq t$ and $p_j \equiv 3~(\mathrm{mod}~4)$ for $t+1 \leq j \leq k$, and $s_j$ be a positive integer for $1 \leq j \leq k$. If $n=2^h{p_1}^{s_1} \cdots {p_t}^{s_t} {p_{t+1}}^{s_{t+1}} \cdots {p_k}^{s_k}$, then $G_n$ is non-integral.
\end{lema}

Now we combine the preceding lemmas to present the following theorem, which classifies all integral quadratic unitary Cayley graphs.
\begin{theorem}\label{Integral QUCG}
The quadratic unitary Cayley graph $G_n$ is integral if and only if $n \in \{2, 4, p^s, 2p^s\}$ for some prime $p$ such that $p \equiv 3~(\mathrm{mod}~4)$ and a positive integer $s$.
\end{theorem}

\begin{corollary}\label{Periodic QUCG}
The quadratic unitary Cayley graph $G_n$ is periodic if and only if $n \in \{2, 4, p^s, 2p^s\}$ for some prime $p$ such that $p \equiv 3~(\mathrm{mod}~4)$ and a positive integer $s$. 
\end{corollary}

Now we determine the quadratic unitary Cayley graphs admitting PST. It follows from Godsil~\cite{State transfer on graphs} [see Theorem 6.1] that if a quadratic unitary Cayley graph $G_n$ admits PST, then it is periodic and $n$ is even. Therefore, if $G_n$ admits PST, then from Corollary~\ref{Periodic QUCG} we obtain that $n \in \{2, 4, 2p^s\}$ for some prime $p$ such that $p \equiv 3~(\mathrm{mod}~4)$ and a positive integer $s$. Clearly, the graphs $G_2$ and $G_4$ admit PST. Theorem~\ref{Tan} implies that $G_{2p^s}$ does not admit PST for any odd prime $p$. Thus, we obtain the following theorem.

\begin{theorem}\label{PST on QUCG}
The quadratic unitary Cayley graph $G_n$ admits PST if and only if $n=2$ or $n=4$.  
\end{theorem}

\section{PGST on quadratic unitary Cayley graphs}\label{sec 5}
In this section, we completely classify the quadratic unitary Cayley graphs admitting PGST. Pal and Bhattacharjya proved the following theorem, which gives a necessary condition for the existence of PGST on circulant graphs.
\begin{theorem}\label{Pal circulant}\emph{\cite{Pal circulant}}
Let $a$ and $b$ be two vertices of $\mathrm{Cay}({\mathbb{Z}}_n, S)$ such that $a < b$. If $\mathrm{Cay}({\mathbb{Z}}_n, S)$ admits PGST from $a$ to $b$, then $n$ is even and $b=a+\frac{n}{2}$.
\end{theorem} 

The following lemma about complex numbers is straightforward.
\begin{lema}\label{PGST lema unit modulus}
Let $n$ be a positive integer and $\phi_0, \ldots, \phi_{n-1}$ be complex numbers of unit modulus. Then $|\phi_0+\cdots+\phi_{n-1}|=n$ if and only if $\phi_0=\cdots=\phi_{n-1}$.
\end{lema}

Now we give a complete classification of $G_n$ admitting PGST. Pal proved the following sufficient condition for the non-existence of PGST on circulant graphs.
\begin{theorem}\emph{\cite{More circulant graphs}}\label{More circulant graphs}
Let $n=mp$, where $p$ is an odd prime and $m$ is an even positive integer. Also, let $S$ be a connection set in ${\mathbb{Z}}_n$ such that  $p \nmid y$ for all $y \in S$. Then the circulant graph $\mathrm{Cay}({\mathbb{Z}}_n, S)$ does not admit PGST.
\end{theorem}
Using Theorem~\ref{More circulant graphs}, we prove the following lemma. 
\begin{lema}\label{QUCG PGST 1}
Let $n=mp$, where $p$ is an odd prime and $m$ is an even positive integer. Then 
$G_n$ does not admit PGST.
\end{lema}
\begin{proof}
Recall that $G_n=\mathrm{Cay}({\mathbb{Z}}_n, T_n)$. Let $y \in T_n$. Then $y=u^2$ or $y=-u^2$ for some $u \in U(n)$. Thus $p \nmid u$ and hence $p \nmid y$. Now the result follows from Theorem~\ref{More circulant graphs}.
\end{proof}

From Theorem~\ref{Pal circulant} and Lemma~\ref{QUCG PGST 1}, it follows that if $G_n$ admits PGST, then it is necessary that $n=2^h$ for some positive integer $h$. We use some of the techniques appeared in~\cite{State transfer on circulant Pal} to prove the last part of the following lemma.
\begin{lema}\label{QUCG PGST 2}
Let $n=2^h$, where $h$ is a positive integer. Then $G_n$ admits PGST if and only if $h \in \{1, 2, 3\}$.
\end{lema}
\begin{proof}
Note that $G_2$ is the path on two vertices that admits PGST. Also, $G_4=C_4$ and $G_8=C_8$. Theorem~\ref{Pal circulant cycle} yields that the graphs $G_4$ and $G_8$ admit PGST. Now we consider the case that $h > 3$. Suppose that the graph $G_n$ admits PGST between the vertices $a$ and $b$ with $a < b$. Then  there exists a real sequence $\{t_k\}$ and a complex number $\gamma$ of unit modulus such that 
$$\lim_{k\to\infty} H(t_k){\textbf{e}_{a}}=\gamma{\textbf{e}_{b}}.$$
This implies that 
$$\lim_{k\to\infty} \sum_{r=0}^{n-1} \exp(-\textbf{i} t_k \lambda_r) {\omega_n}^{(a-b)a_r} =  n\gamma.$$
Theorem~\ref{Pal circulant} gives that $b=a+\frac{n}{2}$. Then we have
$$\lim_{k\to\infty} \sum_{r=0}^{n-1} \exp(-\textbf{i}[t_k \lambda_r+\pi a_r])=n\gamma.$$
Since the unit circle is sequentially compact, there exists a subsequence $\{\exp(-\textbf{i} [t^\prime_k \lambda_r + \pi a_r])\}$ of the sequence $\{\exp(-\textbf{i} [t_k \lambda_r + \pi a_r])\}$ and a complex number $\zeta_r$ of unit modulus for $0 \leq r \leq n-1$ such that  
$$\lim_{k\to\infty} \exp(-\textbf{i}[t^\prime_k \lambda_r+\pi a_r])=\zeta_r.$$
Then we have
\begin{equation}\label{YY}
\sum_{r=0}^{n-1} \zeta_r=n \gamma.
\end{equation}
This implies  
$\left | \sum_{r=0}^{n-1} \zeta_r \right |=n$. Now applying Lemma~\ref{PGST lema unit modulus}, from Equation~(\ref{YY}) we obtain that 
$$\zeta_r=\gamma~~~\mathrm{for}~0 \leq r \leq n-1.$$
Thus we have 
\begin{equation}\label{YYY}
\lim_{k\to\infty} \exp(-\textbf{i}[t^\prime_k \lambda_r+\pi a_r])=\gamma~~~\mathrm{for}~0 \leq r \leq n-1.
\end{equation}
Using Lemma~\ref{lemma 1.2}, Theorem~\ref{theorem character prime power} and Theorem~\ref{theorem canonical factorization}, we obtain that $\lambda_3=0=\lambda_{2^{h-2}}$. Now putting $r=3$ in Equation~(\ref{YYY}), we have $\gamma=-1$. However, $r=2^{h-2}$ gives $\gamma=1$, a contradiction. Thus, the graph $G_n$ does not admit PGST.  
\end{proof}
Combining Lemma~\ref{QUCG PGST 1} and Lemma~\ref{QUCG PGST 2}, we present the following theorem.
\begin{theorem}
The quadratic unitary Cayley graph $G_n$ admits PGST if and only if $n \in \{2, 4, 8\}$. 
\end{theorem}

\section{PGFR on quadratic unitary Cayley graphs}\label{sec 6}
In this section, we obtain a complete classification of quadratic unitary Cayley graphs admitting PGFR in terms of the number of vertices. Recall that if $p$ is an odd prime factor of $n$, then $p \nmid y$ for all $y \in T_n$. Thus, the following lemma follows from Theorem~\ref{PGFR non existence circulant}.
\begin{lema}\label{ Quadratic unitary Cayley graph 2}
Let $n$ be a positive integer such that $2pq \mid n$, where $p$ and $q$ are two distinct odd primes. Then $G_n$ does not admit PGFR.
\end{lema}

\begin{corollary}\label{corollary for quadratic unitary Cayley graph}
If the quadratic unitary Cayley graph $G_n$ admits PGFR, then it is necessary that $n = 2^h p^s$ for some $h \in \mathbb{N}$, $s \in \mathbb{N} \cup \{0\}$ and odd prime $p$. 
\end{corollary}
Clearly, $G_n$ admits PGFR for $n=2$ and $n=4$. Now we 
assume that $n \geq 6$. In the following theorem, we obtain an infinite family of $G_n$ admitting PGFR.
\begin{lema}\label{2p first}
Let $n=2p$, where $p$ is a prime such that $p \equiv 3~(\mathrm{mod}~4)$. Then the graph $G_n$ admits PGFR. 
\end{lema}
\begin{proof}
We write ${\mathbb{Z}}_{2p}={\mathbb{Z}}_{2} \times {\mathbb{Z}}_{p}$. Then we have $\lambda_{(a, b)}=\chi_a(Q_2) \left[\chi_b(Q_p)+\overline{\chi_b(Q_p)} \right]$, where $a \in {\mathbb{Z}}_{2}$ and $b \in {\mathbb{Z}}_{p}$. This implies that $\lambda_0=\lambda_{(0, 0)}=p-1$ and $\lambda_p=\lambda_{(1, 0)}=1-p$. If $r ~(\neq 0)$ is even, then $\lambda_r=\lambda_{(0, b)}=-1$, where $b \in \mathbb{Z}_p \setminus \{0\}$. If $r ~(\neq p)$ is odd, then $\lambda_r=\lambda_{(1, b)}=1$, where $b \in \mathbb{Z}_p \setminus \{0\}$. Now, proceeding as in the proof of Lemma~\ref{unitary Cayley graph 3}, it can be proved that the graph $G_n$ admits PGFR.
\end{proof}

The following theorem provides another infinite family of $G_n$ admitting PGFR.
\begin{lema}\label{2p second}
Let $n=2p$, where $p$ is a prime such that $p \equiv 1~(\mathrm{mod}~4)$. Then $G_n$ admits PGFR. 
\end{lema}
\begin{proof}
We us write ${\mathbb{Z}}_{2p}={\mathbb{Z}}_{2} \times {\mathbb{Z}}_{p}$. Then we have $\lambda_{(a, b)}=\chi_a(Q_2)\chi_b(Q_p)$, where $a \in {\mathbb{Z}}_{2}$ and $b \in {\mathbb{Z}}_{p}$. This implies that $\lambda_0=\lambda_{(0, 0)}=\frac{p-1}{2}$ and $\lambda_p=\lambda_{(1, 0)}=\frac{1-p}{2}$. If $r~(\neq 0)$ is even, then we have
\[\lambda_r=\lambda_{(0,b)}=\left\{ \begin{array}{rl} \frac{-1+\sqrt{p}}{2}  & \textrm{if}~b~\textrm{is a quadratic residue of } p \\
 \frac{-1-\sqrt{p}}{2} & \textrm{if}~b~\textrm{is a quadratic non-residue of } p. ~\\
  \end{array}\right. \]
Also, for an odd $r$ other than $p$, we have 
\[\lambda_r=\lambda_{(1,b)} = \left\{ \begin{array}{rl} \frac{1-\sqrt{p}}{2}  & \textrm{if}~b~\textrm{is a quadratic residue of } p \\
 \frac{1 + \sqrt{p}}{2} & \textrm{if}~b~\textrm{is a quadratic non-residue of } p. ~\\
  \end{array}\right. \]
Define the sets $E_1, E_2, O_1$ and $O_2$ given by
\begin{align*}
E_1&=\{r \colon r~\mathrm{is~even~and}~r \neq 0~\mathrm{such~that}~\lambda_r=\lambda_{(0,b)}~\mathrm{for~some~quadratic~residue~}b~of~p\},\\
E_2&=\{r \colon r~\mathrm{is~even~and}~r \neq 0~\mathrm{such~that}~\lambda_r=\lambda_{(0,b)}~\textrm{for~some~quadratic~non-residue~}b~of~p\},\\
O_1&=\{r \colon r~\mathrm{is~odd~and}~r \neq p~\mathrm{such~that}~\lambda_r=\lambda_{(1,b)}~\mathrm{for~some~quadratic~residue~}b~of~p\}~\mathrm{and}\\
O_2&=\{r \colon r~\mathrm{is~odd~and}~r \neq p~\mathrm{such~that}~\lambda_r=\lambda_{(1,b)}~\textrm{for~some~quadratic~non-residue~}b~of~p\}.
\end{align*}
Suppose that the integers $\ell_1,\ldots,\ell_{n-1}$ satisfy the relation 
\begin{align*}
\sum_{r=1}^{n-1}{\ell_r}(\lambda_r-\lambda_0)=0.
\end{align*}
Then we have 
\begin{align*}
\resizebox{0.999\hsize}{!}{$
-\sum_{r \in E_1}\ell_r-\sum_{r \in E_2}\ell_r+\sum_{r \in O_1}\ell_r+\sum_{r \in O_2}\ell_r+\left(\sum_{r \in E_1}\ell_r-\sum_{r \in E_2}\ell_r-\sum_{r \in O_1}\ell_r+\sum_{r \in O_2}\ell_r \right) \sqrt{p} 
= (p-1)\left(\ell_p+\sum_{r=1}^{n-1}\ell_r \right).$}
\end{align*}
This implies that
\begin{align*}
-\sum_{r \in E_1}\ell_r-\sum_{r \in E_2}\ell_r+\sum_{r \in O_1}\ell_r+\sum_{r \in O_2}\ell_r=(p-1)\left(\ell_p+\sum_{r=1}^{n-1}\ell_r \right).
\end{align*}
This further implies that
\begin{align*}
(2-p)\sum_{r~\mathrm{odd}}\ell_r=p \left(\ell_p+\sum_{r~\mathrm{even}}\ell_r \right).
\end{align*}
Now the rest of the proof is similar to that of Lemma~\ref{unitary Cayley graph 3}. 
\end{proof}
Note that Lemma~\ref{2p first} as well as Lemma~\ref{2p second} provide infinite families of quadratic unitary Cayley graphs admitting PGFR that fail to admit PGST. In the next four lemmas, we prove that $G_n$ does not admit PGFR if $n=4p$ or $n=8p$, where $p$ is an odd prime.

\begin{lema}\label{4p first}
Let $n=4p$, where $p$ is a prime such that $p \equiv 3~(\mathrm{mod}~4)$. Then 
$G_n$ does not admit PGFR. 
\end{lema}
\begin{proof}
We write ${\mathbb{Z}}_{4p}={\mathbb{Z}}_{p} \times {\mathbb{Z}}_{4}$. Then we have $\lambda_{(a, b)}=\chi_a(Q_p)\chi_b(Q_4)+\overline{\chi_a(Q_p)\chi_b(Q_4)}$, where $a \in {\mathbb{Z}}_{p}$ and $b \in {\mathbb{Z}}_{4}$. This implies that $\lambda_{0}=\lambda_{(0, 0)}=p-1$, $\lambda_{3p}=\lambda_{(0, 3)}=0$, $\lambda_{2}=\lambda_{(1, 2)}=1$ and $\lambda_{4}=\lambda_{(1, 0)}=-1$. Define the integers $\ell_1, \ldots, \ell_{n-1}$ such that 
\[{\ell}_r=\left\{ \begin{array}{ll} 1 & \textrm{ if } r=3p\\
 \frac{p-1}{2} & \textrm{ if } r=2~\\
 \frac{1-p}{2} & \textrm{ if } r=4~\\
 0 & \textrm{ otherwise. } \end{array}\right. \]
Then $\sum_{r=1}^{n-1}{\ell_r}(\lambda_r-\lambda_0) = 0$ and $\sum_{r~\mathrm{odd}} \ell_r = 1$. Now, Theorem~\ref{PGFR circulant} yields that $G_n$ does not admit PGFR.
\end{proof}

\begin{lema}\label{4p second}
Let $n=4p$, where $p$ is a prime such that $p \equiv 1~(\mathrm{mod}~4)$. Then 
$G_n$ does not admit PGFR. 
\end{lema}
\begin{proof}
We write ${\mathbb{Z}}_{4p}={\mathbb{Z}}_{p} \times {\mathbb{Z}}_{4}$. Then we have $\lambda_{(a, b)}=\chi_a(Q_p) \left[\chi_b(Q_4)+\overline{\chi_b(Q_4)} \right]$, where $a \in {\mathbb{Z}}_{p}$ and $b \in {\mathbb{Z}}_{4}$. This implies that $\lambda_{0} = \lambda_{(0, 0)}=p-1$ and $\lambda_{1}=\lambda_{(1, 1)}=0$. Define the sets $E_1$ and $E_2$ given by
\begin{align*}
E_1&=\{r \colon r \equiv 0~(\mathrm{mod}~4)~\mathrm{and}~r \neq 0 ~\mathrm{such~that}~\lambda_r=\lambda_{(a,0)}~\mathrm{for~some~quadratic~residue~}a~\mathrm{of}~p\}~\mathrm{and}\\
E_2&=\{r \colon r \equiv 2~(\mathrm{mod}~4)~\mathrm{and}~r \neq 2p ~\mathrm{such~that}~\lambda_r=\lambda_{(a,2)}~\textrm{for~some~quadratic~non-residue~}a~\mathrm{of}~p\}.
\end{align*}

Choose $r_1 \in E_1$ and $r_2 \in E_2$. Then we have $\lambda_{r_1}=-1+\sqrt{p}$ and $\lambda_{r_2}=1+\sqrt{p}$. Define the integers $\ell_1, \ldots, \ell_{n-1}$ such that 
\[{\ell}_r=\left\{ \begin{array}{ll} 1 & \textrm{ if } r=1\\
 \frac{1-p}{2} & \textrm{ if } r=r_1~\\
 \frac{p-1}{2} & \textrm{ if } r=r_2~\\
 0 & \textrm{ otherwise. } \end{array}\right. \]
Then $\sum_{r=1}^{n-1}{\ell_r}(\lambda_r-\lambda_0)=0$ and $\sum_{r~\mathrm{odd}} \ell_r=1$. Therefore, $G_n$ does not admit PGFR.
\end{proof}

\begin{lema}\label{8p first}
Let $n$ be a positive integer such that $n=8p$, where $p$ is a prime with $p \equiv 3~(\mathrm{mod}~4)$. Then the graph $G_n$ does not admit PGFR. 
\end{lema}
\begin{proof}
We write ${\mathbb{Z}}_{8p}={\mathbb{Z}}_{p} \times {\mathbb{Z}}_{8}$. Then we have $\lambda_{(a, b)} = \chi_a(Q_p)\chi_b(Q_8) + \overline{\chi_a(Q_p) \chi_b(Q_8)}$, where $a \in {\mathbb{Z}}_{p}$ and $b \in {\mathbb{Z}}_{8}$. This implies that $\lambda_{0} = \lambda_{(0, 0)}=p-1$ and $\lambda_{3p}=\lambda_{(0, 3)}=\frac{1-p}{\sqrt{2}}$. Define the sets $E_1, E_2, O_1$ and $O_2$ given by
\begin{align*}
E_1&=\{r \colon r \equiv 0~(\mathrm{mod}~8)~\mathrm{and}~r \neq 0 ~\mathrm{such~that}~\lambda_r=\lambda_{(a,0)}~\mathrm{for~some~quadratic~residue~}a~\mathrm{of}~p\},\\
E_2&=\{r \colon r \equiv 4~(\mathrm{mod}~8)~\mathrm{such~that}~\lambda_r=\lambda_{(a,4)}~\textrm{for~some~quadratic~residue~}a~\mathrm{of}~p\},\\
O_1&=\{r \colon r~\mathrm{is~odd~such~that}~\lambda_r=\lambda_{(a,1)}~\textrm{for~some~quadratic~non-residue~}a~\mathrm{of}~p\} \setminus \{p, 3p, 5p, 7p\}~\mathrm{and}\\
O_2&=\{r \colon r~\mathrm{is~odd~such~that}~\lambda_r=\lambda_{(a,3)}~\textrm{for~some~quadratic~non-residue~}a~\mathrm{of}~p\} \setminus \{p, 3p, 5p, 7p\}.
\end{align*}
Choose $r_1 \in E_1, r_2 \in E_2, r_3 \in O_1$ and $r_4 \in O_2$. Then we have 
$$\lambda_{r_1}=- 1,~~\lambda_{r_2}=1,~~\lambda_{r_3}=\frac{-1+\sqrt{p}}{\sqrt{2}}~~\mathrm{and}~~\lambda_{r_4}=\frac{1+\sqrt{p}}{\sqrt{2}}.$$
Define the integers $\ell_1, \ldots, \ell_{n-1}$ such that 
\[{\ell}_r=\left\{ \begin{array}{ll} \frac{1-p}{2}  & \textrm{ if } r=r_1\\
 \frac{p-1}{2} & \textrm{ if } r=r_2~\\
 1 & \textrm{ if } r=3p~\\
  \frac{1-p}{2} & \textrm{ if } r=r_3~\\
 \frac{p-1}{2} & \textrm{ if } r=r_4~\\
 0 & \textrm{ otherwise. } \end{array}\right. \]
Then $\sum_{r=1}^{n-1}{\ell_r}(\lambda_r-\lambda_0)=0$ and $\sum_{r~\mathrm{odd}} \ell_r=1$. Therefore, $G_n$ does not admit PGFR.
\end{proof}

\begin{lema}\label{8p second}
Let $n$ be a positive integer such that $n=8p$, where $p$ is a prime with $p \equiv 1~(\mathrm{mod}~4)$. Then the graph $G_n$ does not admit PGFR. 
\end{lema}
\begin{proof}
We write ${\mathbb{Z}}_{8p}={\mathbb{Z}}_{p} \times {\mathbb{Z}}_{8}$. Then we have $\lambda_{(a, b)}=\chi_a(Q_p) \left[\chi_b(Q_8)+\overline{\chi_b(Q_8)} \right]$, where $a \in {\mathbb{Z}}_{p}$ and $b \in {\mathbb{Z}}_{8}$. This implies that $\lambda_{0}=\lambda_{(0, 0)}=p-1$ and $\lambda_{3p}=\lambda_{(0, 3)}=\frac{1-p}{\sqrt{2}}$. Define the sets $E_1, E_2, O_1$ and $O_2$ given by
\begin{align*}
E_1&=\{r \colon r \equiv 0~(\mathrm{mod}~8)~\mathrm{and}~r \neq 0 ~\mathrm{such~that}~\lambda_r=\lambda_{(a,0)}~\mathrm{for~some~quadratic~residue~}a~\mathrm{of}~p\},\\
E_2&=\{r \colon r \equiv 4~(\mathrm{mod}~8)~\mathrm{such~that}~\lambda_r=\lambda_{(a,4)}~\textrm{for~some~quadratic~non-residue~}a~\mathrm{of}~p\},\\
O_1&=\{r \colon r~\mathrm{is~odd~such~that}~\lambda_r=\lambda_{(a,1)}~\mathrm{for~some~quadratic~residue~}a~\mathrm{of}~p\} \setminus \{p, 3p, 5p, 7p\}~\mathrm{and}\\
O_2&=\{r \colon r~\mathrm{is~odd~such~that}~\lambda_r=\lambda_{(a,3)}~\textrm{for~some~quadratic~non-residue~}a~\mathrm{of}~p\} \setminus \{p, 3p, 5p, 7p\}.
\end{align*}
Choose $r_1 \in E_1, r_2 \in E_2, r_3 \in O_1$ and $r_4 \in O_2$. Then we have 
$$\lambda_{r_1}=-1+\sqrt{p},~~\lambda_{r_2}=1+\sqrt{p},~~\lambda_{r_3} = \frac{-1+\sqrt{p}}{\sqrt{2}}~~\mathrm{and}~~\lambda_{r_4} = \frac{1+\sqrt{p}}{\sqrt{2}}.$$
Define the integers $\ell_1, \ldots, \ell_{n-1}$ such that 
\[{\ell}_r=\left\{ \begin{array}{ll} \frac{1-p}{2}  & \textrm{ if } r=r_1\\
 \frac{p-1}{2} & \textrm{ if } r=r_2~\\
 1 & \textrm{ if } r = 3p~\\
  \frac{1-p}{2} & \textrm{ if } r=r_3~\\
 \frac{p-1}{2} & \textrm{ if } r=r_4~\\
 0 & \textrm{ otherwise. } \end{array}\right. \]
Then $\sum_{r=1}^{n-1}{\ell_r}(\lambda_r-\lambda_0)=0$ and $\sum_{r~\mathrm{odd}} \ell_r=1$. Therefore, $G_n$ does not admit PGFR.
\end{proof}

\begin{lema}\label{h > 2}
Let $h$ be an integer with $h > 2$ and $n=2^h$. Then $G_n$ admits PGFR if and only if $h=3$.
\end{lema}
\begin{proof}
For $h=3$, we obtain that $G_8=C_8$. Since $C_8$ admits PGST, it admits PGFR as well. Now we assume that $h > 3$. From the proof of Lemma~\ref{QUCG PGST 2}, recall that $\lambda_3=0=\lambda_{2^{h-2}}$. Consider the integers ${\ell}_1, \ldots, {\ell}_{n-1}$ such that 
\[{\ell}_r=\left\{ \begin{array}{rl} 1 & \textrm{ if } r=3\\
 -1 & \textrm{ if } r=2^{h-2}~\\
 0 & \textrm{ otherwise. } \end{array}\right. \]
Then $\sum_{r=1}^{n-1}{\ell_r}(\lambda_r-\lambda_0)=0$ and $\sum_{r~\mathrm{odd}} \ell_r=1$. Therefore, $G_n$ does not admit PGFR. This completes the proof.
\end{proof}

\begin{lema}
Let $h$ be a positive integer such that $h > 3$ and $p$ be a prime with $p \equiv 3~(\mathrm{mod}~4)$. If $n=2^h p$, then $G_n$ does not admit PGFR. 
\end{lema}
\begin{proof}
We write ${\mathbb{Z}}_{2^h p}={\mathbb{Z}}_{2^h} \times {\mathbb{Z}}_{p}$. Then we have $\lambda_{(a, b)}=\chi_a(Q_{2^h})\chi_b(Q_p)+\overline{\chi_a(Q_{2^h})\chi_b(Q_p)}$, where $a \in {\mathbb{Z}}_{2^h}$ and $b \in {\mathbb{Z}}_{p}$. This implies that $\lambda_{1}=\lambda_{(1, 1)}=0$ and $\lambda_{2^{h-2}p}=\lambda_{(2^{h-2}, 0)}=0$. Define the integers $\ell_1, \ldots, \ell_{n-1}$ such that 
\[{\ell}_r=\left\{ \begin{array}{rl} 1 & \textrm{ if } r=1\\
 -1 & \textrm{ if } r=2^{h-2}p~\\
 0 & \textrm{ otherwise. } \end{array}\right. \]
Then $\sum_{r=1}^{n-1}{\ell_r}(\lambda_r-\lambda_0)=0$ and $\sum_{r~\mathrm{odd}} \ell_r=1$. Therefore, $G_n$ does not admit PGFR.
\end{proof}

\begin{lema}
Let $h$ be an integer such that $h > 3$ and $p$ be a prime with $p \equiv 1~(\mathrm{mod}~4)$. If $n=2^h p$, then $G_n$ does not admit PGFR. 
\end{lema}
\begin{proof}
We write ${\mathbb{Z}}_{2^h p}={\mathbb{Z}}_{p} \times {\mathbb{Z}}_{2^h}$. Then we have $\lambda_{(a, b)}=\chi_a(Q_{p}) \left[\chi_b(Q_{2^h})+\overline{\chi_b(Q_{2^h})} \right]$, where $a \in {\mathbb{Z}}_{p}$ and $b \in {\mathbb{Z}}_{2^h}$. This implies that $\lambda_{1}=\lambda_{(1, 1)}=0$ and $\lambda_{2^{h-2}}=\lambda_{(1, 2^{h-2})}=0$. Define the integers $\ell_1, \ldots, \ell_{n-1}$ such that 
\[{\ell}_r=\left\{ \begin{array}{rl} 1 & \textrm{ if } r=1\\
 -1 & \textrm{ if } r=2^{h-2}~\\
 0 & \textrm{ otherwise. } \end{array}\right. \]
Then $\sum_{r=1}^{n-1}{\ell_r}(\lambda_r-\lambda_0)=0$ and $\sum_{r~\mathrm{odd}} \ell_r=1$. Therefore, $G_n$ does not admit PGFR.
\end{proof}

\begin{lema}\label{2^h p^s}
Let $h,s$ be positive integers such that $s>1$ and $p$ be a prime such that $p \equiv 3~(\mathrm{mod}~4)$. If $n=2^h p^s$, then $G_n$ does not admit PGFR. 
\end{lema}
\begin{proof}
We write ${\mathbb{Z}}_{2^h p^s}={\mathbb{Z}}_{2^h} \times {\mathbb{Z}}_{p^s}$. Then we consider the following two cases. 

\noindent \textbf{Case 1.} If $h=1$, then $\lambda_{(a, b)}=\chi_a(Q_{2})\left[\chi_b(Q_{p^s})+\overline{\chi_b(Q_{p^s})} \right]$, where $a \in {\mathbb{Z}}_{2}$ and $b \in {\mathbb{Z}}_{p^s}$. This implies that $\lambda_{1}=\lambda_{(1, 1)}=0$ and $\lambda_{2}=\lambda_{(0, 1)}=0$. Define the integers $\ell_1, \ldots, \ell_{n-1}$ such that 
\[{\ell}_r=\left\{ \begin{array}{rl} 1 & \textrm{ if } r=1\\
 -1 & \textrm{ if } r=2~\\
 0 & \textrm{ otherwise. } \end{array}\right. \]
Then $\sum_{r=1}^{n-1}{\ell_r}(\lambda_r-\lambda_0)=0$ and $\sum_{r~\mathrm{odd}} \ell_r=1$, and therefore $G_n$ does not admit PGFR.

\noindent \textbf{Case 2.} If $h > 1$, then $\lambda_{(a, b)}=\chi_a(Q_{2^h})\chi_b(Q_{p^s}+\overline{\chi_a(Q_{2^h}) \chi_b(Q_{p^s})}$, where $a \in {\mathbb{Z}}_{2^h}$ and $b \in {\mathbb{Z}}_{p^s}$. This implies that $\lambda_{1}=\lambda_{(1, 1)}=0$ and $\lambda_{2^h}=\lambda_{(0, 1)}=0$. Now the rest of the proof is similar to that of the preceding case. 
\end{proof}

The proof of the next lemma is similar to that of Lemma~\ref{2^h p^s}.
\begin{lema}\label{last lemma}
Let $h,s$ be positive integers such that $s>1$ and $p$ be a prime such that $p \equiv 1~(\mathrm{mod}~4)$. If $n=2^h p^s$, then $G_n$ does not admit PGFR. 
\end{lema}

Now we summarize the preceding lemmas to conclude the main result of this section.
\begin{theorem}\label{main theorem PGFR QUCG}
The quadratic unitary Cayley graph $G_n$ admits PGFR if and only if $n \in \{2, 8, 2p\}$ for some prime $p$.  
\end{theorem}

\section{FR on quadratic unitary Cayley graphs}\label{sec 7}
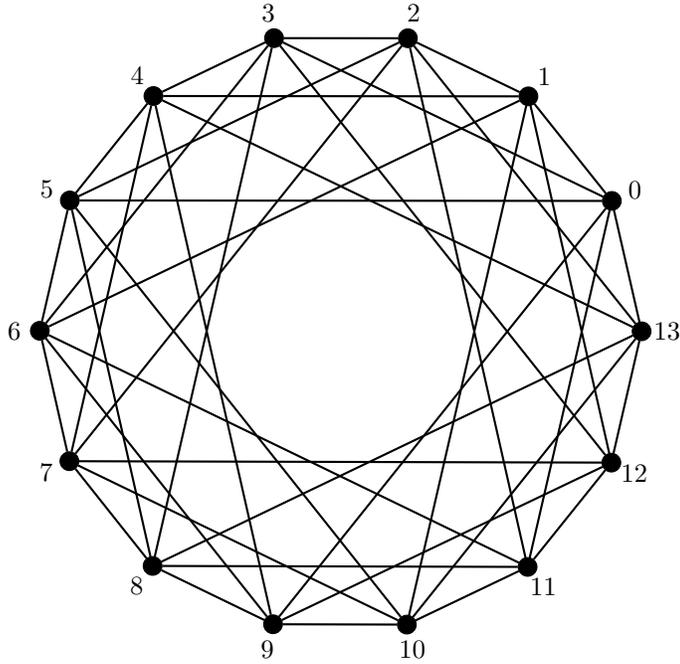
\begin{figure}
\begin{center}
	\begin{tikzpicture}[scale=2,auto,swap]
    \tikzstyle{blackvertex}=[circle,draw=black,fill=black]
    \tikzstyle{bluevertex}=[circle,draw=blue,fill=blue]
    \tikzstyle{greenvertex}=[circle,draw=green,fill=green]
    \tikzstyle{yellowvertex}=[circle,draw=yellow,fill=yellow]
    \tikzstyle{brownvertex}=[circle,draw=brown,fill=brown]
    \tikzstyle{redvertex}=[circle,draw=red,fill=red]
    \tikzstyle{blackvertex}=[circle,draw=black,fill=black]
    \tikzstyle{bluevertex}=[circle,draw=blue,fill=blue]
    \tikzstyle{greenvertex}=[circle,draw=green,fill=green]
    \tikzstyle{yellowvertex}=[circle,draw=yellow,fill=yellow]
    \tikzstyle{brownvertex}=[circle,draw=brown,fill=brown]
    \tikzstyle{redvertex}=[circle,draw=red,fill=red]
    \tikzstyle{brownvertex}=[circle,draw=brown,fill=brown]
    \tikzstyle{redvertex}=[circle,draw=red,fill=red]

    \node [blackvertex,scale=0.75] (a0) at (25.7:2) {};
    \node [blackvertex,scale=0.75] (a1) at (51.4:2) {};
    \node [blackvertex,scale=0.75] (a2) at (77.1:2) {};
    \node [blackvertex,scale=0.75] (a3) at (102.8:2) {};
    \node [blackvertex,scale=0.75] (a4) at (128.5:2) {};
    \node [blackvertex,scale=0.75] (a5) at (154.2:2) {};
    \node [blackvertex,scale=0.75] (a6) at (179.9:2) {};
    \node [blackvertex,scale=0.75] (a7) at (205.6:2) {};
    \node [blackvertex,scale=0.75] (a8) at (231.3:2) {};
    \node [blackvertex,scale=0.75] (a9) at (257:2) {};
    \node [blackvertex,scale=0.75] (a10) at (282.7:2) {};
    \node [blackvertex,scale=0.75] (a11) at (308.4:2) {};
    \node [blackvertex,scale=0.75] (a12) at (334.1:2) {};
    \node [blackvertex,scale=0.75] (a13) at (0:2) {};

    \node [scale=1] at (25.7:2.17) {0};
    \node [scale=1] at (51.4:2.17) {1};
    \node [scale=1] at (77.1:2.17) {2};
    \node [scale=1] at (102.8:2.17) {3};
    \node [scale=1] at (128.5:2.17) {4};
    \node [scale=1] at (154.2:2.17) {5};
    \node [scale=1] at (179.9:2.17) {6};
    \node [scale=1] at (205.6:2.17) {7};
    \node [scale=1] at (231.3:2.17) {8};
    \node [scale=1] at (257:2.17) {9};
    \node [scale=1] at (282.7:2.17) {10};
    \node [scale=1] at (308.4:2.17) {11};
    \node [scale=1] at (334.1:2.17) {12};
    \node [scale=1] at (0:2.17) {13};

\draw [black,thick] (a0) -- (a1) -- (a2) -- (a3) -- (a4) -- (a5) -- (a6) --(a7) -- (a8) -- (a9) -- (a10) -- (a11) -- (a12) -- (a13) -- (a0);
\draw [black,thick] (a0) -- (a3);
\draw [black,thick] (a0) -- (a5);
\draw [black,thick] (a0) -- (a9);
\draw [black,thick] (a0) -- (a11);
\draw [black,thick] (a1) -- (a4);
\draw [black,thick] (a1) -- (a6);
\draw [black,thick] (a1) -- (a10);
\draw [black,thick] (a1) -- (a12);
\draw [black,thick] (a2) -- (a5);
\draw [black,thick] (a2) -- (a7);
\draw [black,thick] (a2) -- (a11);
\draw [black,thick] (a2) -- (a13);
\draw [black,thick] (a3) -- (a6);
\draw [black,thick] (a3) -- (a8);
\draw [black,thick] (a3) -- (a12);
\draw [black,thick] (a4) -- (a7);
\draw [black,thick] (a4) -- (a9);
\draw [black,thick] (a4) -- (a13);
\draw [black,thick] (a5) -- (a8);
\draw [black,thick] (a5) -- (a10);
\draw [black,thick] (a6) -- (a9);
\draw [black,thick] (a6) -- (a11);
\draw [black,thick] (a7) -- (a10);
\draw [black,thick] (a7) -- (a12);
\draw [black,thick] (a8) -- (a11);
\draw [black,thick] (a8) -- (a13);
\draw [black,thick] (a9) -- (a12);
\draw [black,thick] (a10) -- (a13);

\end{tikzpicture}
\caption{Quadratic unitary Cayley graph admitting FR but not admitting PST}
\label{figure2}
\end{center}
\end{figure}

In this section, we obtain all the quadratic unitary Cayley graphs admitting FR. From Theorem~\ref{main theorem PGFR QUCG}, it is clear that if $G_n$ admits FR, then $n \in \{2, 8, 2p\}$ for some prime $p$. The graphs $G_2$ and $G_4$ admit FR. By Theorem~\ref{Chan FR}, the graph $G_8$ does not admit FR. For an odd prime $p$, we now prove that the graph $G_{2p}$ admits FR only when $p \equiv 3~(\mathrm{mod}~4)$.

\begin{lema}\label{2p first FR}
Let $n=2p$, where $p$ is a prime such that $p \equiv 3~(\mathrm{mod}~4)$. Then $G_n$ admits FR.
\end{lema}
\begin{proof}
From Lemma~\ref{2p first}, recall that the eigenvalues of $G_n$ are given by
\[{\lambda}_r=\left\{ \begin{array}{ll} 
p-1 & \textrm{ if } r=0\\
-1 & \textrm{ if } r~\mathrm{is~even~and}~r \neq 0\\
 1-p & \textrm{ if } r=p~\\
 1 & \textrm{ otherwise. } \end{array}\right. \]
Then, following the exact procedure as in the proof of Theorem~\ref{UCG FR}, it can be proved that the graph $G_n$ admits FR. This completes the proof.
\end{proof}

\begin{lema}\label{2p second FR}
Let $n$ be a positive integer such that $n=2p$, where $p$ is a prime with $p \equiv 1~(\mathrm{mod}~4)$. Then the graph $G_n$ does not admit FR.
\end{lema}
\begin{proof}
Consider the sets $E_1, E_2, O_1$ and $O_2$ from the proof of Lemma~\ref{2p second}. Let $x_1 \in O_1$ and $y_1 \in O_2$ such that $x_1 > y_1$. Also, let $x_2 \in E_1$ and $y_2=0$.  Then we have
$${\lambda}_{x_1}=\frac{1-\sqrt{p}}{2},~~~{\lambda}_{y_1}=\frac{1+\sqrt{p}}{2},~~~ {\lambda}_{x_2}=\frac{-1+\sqrt{p}}{2}~~~\mathrm{and}~~~\lambda_{y_2} = \frac{p-1}{2}.$$
Now 
$$\frac{\lambda_{x_1}-\lambda_{y_1}}{\lambda_{x_2}-\lambda_{y_2}} = \frac{2}{p-1}+\frac{2}{p-1} \sqrt{p} \notin \mathbb{Q}.$$
Therefore, from Corollary~\ref{non existence of FR} we obtain that $G_n$ does not admit FR. This completes the proof.
\end{proof}
We note that Lemma~\ref{2p second FR} also follows from Cao and Luo~\cite{Cao} [see Theorem 2.2]. The preceding results are combined in the next theorem.
\begin{theorem}\label{main theorem for FR quadratic unitary Cayley graph}
The quadratic unitary Cayley graph $G_n$ admits FR if and only if $n \in \{2, 4, 2p\}$ for some prime $p$ such that $p \equiv 3~(\mathrm{mod}~4)$.
\end{theorem}

We conclude by noting that Theorem~\ref{PST on QUCG} and Lemma~\ref{2p first FR}  altogether provide an infinite family of quadratic unitary Cayley graphs admitting FR that fail to admit PST. Also, Lemma~\ref{2p second} and Lemma~\ref{2p second FR} together provide an infinite family of quadratic unitary Cayley graphs admitting PGFR that fail to admit FR. In particular, the graph $G_{14}$ admits FR but does not admit PST. This graph is shown in Figure~\ref{figure2}.

\subsection*{Acknowledgements} The first author acknowledges funding received from the
Prime Minister’s Research Fellowship (PMRF), PMRF-ID: 1903283, Ministry of Education, Government of India, for carrying out this research work.

\end{document}